\documentclass[a4paper,11pt,reqno]{article}

\usepackage[utf8]{inputenc}
\usepackage[T1]{fontenc}
\usepackage{lmodern}
\usepackage[english]{babel}
\usepackage{microtype}

\usepackage{amsmath,amssymb,amsfonts,amsthm}
\usepackage{mathtools,accents}
\mathtoolsset{showonlyrefs} 
\usepackage{mathrsfs}
\usepackage{aliascnt}
\usepackage{braket}
\usepackage{bm}
\usepackage{esint}
\usepackage{centernot}
\usepackage{enumitem}
\usepackage{changepage}
\usepackage[a4paper,margin=3cm]{geometry}

\usepackage{hyperref}
\hypersetup{%
	colorlinks = true,
	linkcolor  = black,
	citecolor = blue
}

\usepackage{xcolor}

\makeatletter
\g@addto@macro\@floatboxreset\centering
\makeatother

\makeatletter
\def\newaliasedtheorem#1[#2]#3{
  \newaliascnt{#1@alt}{#2}
  \newtheorem{#1}[#1@alt]{#3}
  \expandafter\newcommand\csname #1@altname\endcsname{#3}
}
\makeatother

\numberwithin{equation}{section}

\newtheoremstyle{slanted}{\topsep}{\topsep}{\slshape}{}{\bfseries}{.}{.5em}{}

\theoremstyle{plain}
\newtheorem{thm}{Theorem}[section]
\newaliasedtheorem{prop}[thm]{Proposition}
\newaliasedtheorem{lem}[thm]{Lemma}
\newaliasedtheorem{cor}[thm]{Corollary}
\newaliasedtheorem{counterexample}[thm]{Counterexample}

\theoremstyle{definition}
\newaliasedtheorem{defn}[thm]{Definition}
\newaliasedtheorem{question}[thm]{Question}
\newaliasedtheorem{example}[thm]{Example}
\newaliasedtheorem{conjecture}[thm]{Conjecture}
\newaliasedtheorem{statement}[thm]{Statement}
\newaliasedtheorem{observation}[thm]{Observation}

\theoremstyle{remark}
\newaliasedtheorem{rem}[thm]{Remark}
\providecommand{\keywords}[1]{{\fontsize{10pt}{1em}\textbf{Keywords: } \textit{#1.}}}
\providecommand{\ams}[1]{{\fontsize{10pt}{1em}\textbf{2010 Mathematics Subject Classification: } \textit{#1.}}}

\newcommand{\setR}{\mathbb{R}}

\let\altphi\phi
\let\phi\varphi
\let\varphi\altphi
\let\altphi\undefined

\newcommand{\haus}{\mathcal{H}}
\newcommand{\leb}{\mathcal{L}}

\newcommand{\measrestr}{%
  \,\raisebox{-.127ex}{\reflectbox{\rotatebox[origin=br]{-90}{$\lnot$}}}\,%
}

\DeclareMathOperator{\Tr}{Tr}

\newcommand{\di}{\mathop{}\!\mathrm{d}}

\newcommand{\df}{\vcentcolon =}

\newcommand{\overbar}[1]{\mkern 1.5mu\overline{\mkern-1.5mu#1\mkern-1.5mu}\mkern 1.5mu}
\newcommand{\uunderbar}[1]{\mkern 1.5mu\underline{\mkern-1.5mu#1\mkern-1.5mu}\mkern 1.5mu}
\DeclareMathOperator{\sgn}{sgn}

\newcommand{\dist}{\mathsf{d}}

\newcommand{\meas}{\mu}

\DeclareMathOperator{\vol}{\mathrm{vol}}

\newfont{\tmpf}{cmsy10 scaled 2500}

\begin{document}
\title{Asymptotic Mean Value Laplacian in Metric Measure Spaces}
\author{Andreas Minne
\thanks{KTH Royal Institute of Technology, \url{minne@kth.se}.} \and
David Tewodrose
\thanks{Université de Cergy-Pontoise, \url{david.tewodrose@u-cergy.fr}}} 
\maketitle

\begin{abstract}
We use the mean value property in an asymptotic way to provide a notion of a pointwise Laplacian, called AMV Laplacian, that we study in several contexts including the Heisenberg group and weighted Lebesgue measures. We focus especially on a class of metric measure spaces including intersecting submanifolds of $\setR^n$, a context in which our notion brings new insights; the Kirchhoff law appears as a special case. In the general case, we also prove a maximum and comparison principle, as well as a Green-type identity for a related operator.
\end{abstract}

\begin{adjustwidth}{1cm}{1cm}
\noindent
\keywords{harmonic function, mean value property, metric measure space, maximum principle}
\newline
\newline
\noindent
\ams{Primary: 31C05; Secondary: 35B50, 28E99, 26B15}
\end{adjustwidth}

\tableofcontents

\section{Introduction and Definitions}

Harmonic functions defined on Euclidean domains are well-known to have the mean value property: for any domain $\Omega \subset \setR^n$, a function $u\in C^2(\Omega)$ such that $\Delta u =0 $ (where $\Delta =\sum_{k=1}^n \partial_{kk}$) satisfies
\begin{equation}\label{eq:strongharm}
u(x) = \fint_{B_r(x)} u   
\end{equation}
for any solid ball $B_r(x)\subseteq \Omega$ centered at $x$ with radius $r>0$. We call this the \textit{strong mean value property}. Here, $\fint_{B_r(x)} u$ is defined as the average integral $\leb^n(B_r(x))^{-1}\int_{B_r(x)} u \di \leb^n$, where $\leb^n$ denotes the Lebesgue measure. This celebrated result is credited to Gauss \cite{gauss}. 

The converse implication was first studied by Koebe \cite{koebe}. He proved that if $u$ is continuous in $\Omega$ and satisfies the mean value property on every sphere, then $u$ is harmonic, i.e., $\Delta u=0$. Moreover, Koebe's arguments show that the same statement is true if the mean value property is satisfied at every $x \in \Omega$ only for some radii $\{r_i(x)\}_i$ such that $\inf_i r_i(x)=0$. Later on, Volterra \cite{volterra} for regular domains $\Omega$  and then Kellog \cite{kellogg} for general domains  proved that \eqref{eq:strongharm} is enough for a \emph{single} radius, i.e., a function $u \in C(\overline{\Omega})$ is harmonic if it satisfies
\begin{equation}\label{eq:weakharm}
u(x)=\fint_{B_{r(x)}(x)} u \qquad \forall x \in \Omega,
\end{equation}
where $r$ is a positive function on $\Omega$ with $r(x)<\dist(x,\Omega^c)$. We call this the \textit{weak mean value property}. To sum up, under suitable conditions, we have the following equivalences.
\begin{align*}\label{eq:Eucideanequivalence}
  \text{$u$ is harmonic.}	\iff\quad	&  \text{$u$ has the strong mean value property.}\\
	\iff\quad	&  \text{$u$ has the weak mean value property.}
\end{align*}
For the interested reader, we refer to the exhaustive survey by Netuka and Veselý \cite{MR1321628}, and the work of Llorente \cite{MR3299033}.

Moving from Euclidean domains to a general metric measure space $(X,\dist,\meas)$, the above picture led Gaczkowski and Górka \cite{MR2545982} to study the properties of locally integrable functions $u:X\to\setR$ such that
\begin{equation}\label{eq:strongharmonic}
u(x) = \fint_{B_r(x)} u \di \meas
\end{equation}
for any ball $B_r(x)\subseteq \Omega$ centered at $x$ with radius $r>0$.
Such functions are called \textit{strongly harmonic} by Adamowicz, Gaczkowski and Górka \cite{MR3896674}. Also \textit{weakly harmonic} functions are introduced, namely those locally integrable functions $u:X\to\setR$ satisfying, for every $x$, the mean value property for a single radius denoted $r(x)$:
\begin{equation}\label{eq:weakharmonic}
u(x) = \fint_{B_{r(x)}(x)} u \di \meas.
\end{equation}
It turns out that, in this context in which a pointwise definition of a Laplacian is delicate to set, strongly and weakly harmonic functions share some common properties with harmonic functions on Euclidean domains. For instance, the maximum principle and the Harnack inequality hold under rather general assumptions.

However, weak harmonicity does not imply strong harmonicity \cite[Example 1]{MR3896674}:
\[
\text{$u$ is weakly harmonic.}	\centernot \implies  \text{$u$ is strongly harmonic.}
\]
Moreover, as shown by Bose \cite{MR177128}, there exist functions in $\setR^n$ satisfying a weighted Laplace equation $L_w = 0$ without being strongly harmonic.
In hindsight this might not be too surprising due to $\mathbb{R}^n$ being homogeneous seen as a medium. For a general measure, local properties (think partial derivatives) might have little to do with macroscopic properties (think strong harmonicity). We will return to this issue in Section \ref{sec:examples}.

In this article, we propose another approach based on the asymptotic fulfillment of the mean value property. This idea has been in the air for some time --- see for instance the works of Manfredi et al.~\cite{MR2566554, MR3177660} and of Burago et al.~\cite{MR3990939} --- but perhaps has not been completely crystallized yet.

We want to point out that during the finalization of this article, it came to our attention via private communication with Adamowicz that he, Kijowski and Soultanis \cite{aks} independently have come up with the same definition. They have investigated several related problems, e.g.~the Hölder regularity of continuous asymptotic mean value harmonic functions, and the dimension of the space of continuous asymptotic mean value harmonic harmonic functions with polynomial growth.

It should be noted that when we refer to a metric measure space $(X,\meas,\dist)$, we always consider a strictly positive metric measure space, i.e. $\meas(\Omega)>0$ for every non-empty open subset $\Omega\subset X$. To be able to discuss pointwise properties of $L^1_{\text{loc}}$ functions, we use the convention of choosing the representative $\tilde{u}$ of $u\in L^1_{\text{loc}}(X,\meas)$ defined at $x$ as the limit of $\fint_{B_r(x)} u(y)\di \meas(y)$ as $r\to 0^+$, whenever the limit exists, which it does $\meas$-a.e.

\begin{defn}[AMV Laplacian]
Let $(X,\dist,\meas)$ be a metric measure space and $u:X\to\overline{\setR}$  be locally integrable. Then the \emph{asymptotic mean value Laplacian} (AMV Laplacian for short) of $u$ is defined as
 $$
 \Delta_\meas^{\dist} u(x)\df \lim\limits_{r \to 0^+} \frac{1}{r^2} \fint_{B_r(x)}u(y) - u(x)\di \meas(y)
 $$
for any $x \in X$ for which the limit exists. We also define for $\meas$-a.e.~$x$
$$
 \Delta_{\mu,r}^{\dist} u(x)\df \frac{1}{r^2} \fint_{B_r(x)}u(y) - u(x)\di \meas(y)
$$
for any $r>0$.
\end{defn}

If $u$ is defined only on a subset $\Omega \subset X$, then we set
$$
 \Delta_{\mu,r}^{\dist} u(x)\df \frac{1}{r^2} \fint_{B_r(x)\cap \Omega}u(y) - u(x)\di \meas(y).
$$
Note that this definition does not require $\Omega$ to be open, only that $\meas(B_r(x)\cap \Omega)>0$ for all positive $r$.

Having a notion of pointwise Laplacian at our disposal, we can define harmonicity in the following manner.

\begin{defn}[AMV Harmonic Function]\label{def:AMVharmonic}
Let $(X,\dist,\mu)$ be a metric measure space and $u\in L^1_{\text{loc}}(X,\meas)$. We say that $u$ is \emph{asymptotically mean value harmonic} (AMV harmonic for short) in $\Omega\subset X$ if $\Delta_\mu^{\dist} u(x) = 0$ for all $x \in \Omega$.
\end{defn}

\begin{rem}\label{rem:subsuper}
 There are evident versions $\overbar{\Delta}_{\mu,r}^{\dist} u(\cdot)$ and $\uunderbar{\Delta}_{\mu,r}^{\dist} u(\cdot)$ in which the limit is replaced by $\limsup$ and $\liminf$, for which one can define asymptotically mean value sub- and superharmonic functions, see Section \ref{sec:maximum}.
\end{rem}

AMV harmonic functions have some advantages over strongly harmonic ones. Indeed, the latter are trivially seen to be AMV harmonic. Moreover, AMV harmonicity comes along with a natural notion of Laplacian, which is absent in the context of strongly harmonic functions. Therefore, it is possible to consider the corresponding Poisson equation, $\Delta_\mu^{\dist} u(x) = f$, heat equation or Helmholtz equation, etc.

A natural question is the relation of the AMV Laplacian with other notions of Laplacians. When $u \in C^2(\setR^n)$, one can easily show that a second order Taylor expansion gives
\begin{equation}\label{eq:euclidean}
\Delta_{\leb^n}^{\dist_e}u = \frac{1}{2(n+2)} \Delta u,
\end{equation}
where $\dist_e$ stands for the Euclidean distance. Using normal coordinates, this calculation can be adapted to the setting of a $C^2$ Riemannian manifold $(M,g)$ in order to prove that for any $u \in C^2(M)$ and any interior point $x \in M$,
\begin{equation}\label{eq:manifold}
\Delta_{\vol_g}^{\dist_g}u(x) = \frac{1}{2(n+2)} \Delta_g u(x),
\end{equation}
where $\dist_g$, $\vol_g$ and $\Delta_g$ are the canonical Riemannian distance, the Riemannian volume measure and the Laplace--Beltrami operator of $(M,g)$ respectively. In Section \ref{sec:examples}, we consider more examples: the Heisenberg group, and the Euclidean space equipped with a weighted Lebesgue measure, showing that in both cases the AMV Laplacian is comparable to the corresponding Laplace operator, namely the Kohn Laplacian and the weighted Laplacian respectively.

More interestingly, our definition permits to deduce some results on spaces which, to the best of our knowledge, has not been proposed yet. We focus especially on metric measure spaces for which the measure has a dimension that depends on the part of the space we are looking at. This is formalized through the notion of a stratified measure, see Section \ref{sec:stratified}. Our results (Theorem \ref{thm:intersection-ahlfors-regular-measures} and Corollary \ref{cor:intersection-ahlfors-regular-measures-nondecreasing}) show that the AMV Laplacian at a point in the intersection of the supports of measures with different Ahlfors dimension only takes into account the \textit{lowest} dimension. Here by Ahlfors dimension we mean a number $Q\ge 0$ such that the measure is $Q$-Ahlfors regular, see Section \ref{sec:stratified} for details. We apply these results in the context of submanifolds of $\setR^n$ intersecting each other, see Corollary \ref{cor:???}. Note that our example in Subsection \ref{euclideandirac}, namely the Euclidean space equipped with the Lebesgue measure plus a Dirac mass, can be seen as a particular case of these stratified spaces.

In Section \ref{sec:maximum} we define AMV sub- and superharmonic functions, and show that the maximum of upper semicontinuous AMV subharmonic functions is attained on the boundary (see Theorem~\ref{thm:weak-maxprin-sub} for precise assumptions). A symmetrical argument goes through for lower semicontinuous AMV superharmonic functions, and a comparison principle is obtained as a corollary.

Finally, we prove a Green-type identity for the operators $\Delta_{\mu,r}^{\dist}$ restricted to a suitable weighted $L^2$ space. This formula suggests to define the weak AMV Laplacian $\Delta_\meas^\dist u$ as the measure $\nu$ such that 
\[
\int_X \phi \di \nu = \lim\limits_{r\to 0^+} \int_X \phi \Delta_{\meas,r}^\dist u   \di \mu
\]
 holds for any $\phi \in C_c(X)$, see Definition~\ref{def:weakAMV}. It is worth pointing out that with this definition, pointwise AMV harmonic functions might fail to be weakly AMV harmonic. An example can be found in the paragraph following Definition~\ref{def:weakAMV}.

\hfill

\smallskip\noindent
\textbf{Acknowledgements.}
 A.~Minne was supported by the Knut and Alice Wallenberg Foundation, as well as Stiftelsen G S Magnusons fond. We are both grateful to Scuola Normale Superiore di Pisa at which most of this work was conducted, and to T.~Adamowicz for his invitation to IMPAN where we had inspiring final discussions with him, A.~Kijowski, and E.~Soultanis.


\section{Examples}\label{sec:examples}

In this section, we get some familiarity with the AMV Laplacian by looking at three different examples for which it is possible to do explicit computations.

\subsection{Heisenberg group}

Let $\mathbb{H}$ be the Heisenberg group that we interpret here as $\setR^3$ equipped with the follwing group law:
\[
(x,y,t) \circ (x',y',t') \df (x+x',y+y',t+t'+2(yx' - y'x)) \qquad \forall (x,y,t), (x',y',t') \in \setR^3.
\]
We equip $\mathbb{H}$ with the classical vector fields
\[
X=\frac{\partial}{\partial x} + 2y \frac{\partial}{\partial t}, \qquad Y = \frac{\partial}{\partial y} - 2x \frac{\partial}{\partial t}, \qquad T=\frac{\partial}{\partial t},
\]
which provide a sub-Riemannian structure to $\mathbb{H}$. We denote by $\dist_{CC}$ the associated intrinsic metric, called the Carnot--Carathéodory metric (see e.g.~\cite[5.2]{MR2363343}). We recall that $B_r(p)=p\circ\delta_r(B_1(o))$ for any $p \in \mathbb{H}$ and any $r>0$, where $\delta_r$ is the dilation $(x,y,t) \mapsto (rx,ry,r^2t)$ and $o=(0,0,0)$ is the origin.

For any $p_o=(x_o,y_o,t_o) \in \mathbb{H}$, let $L_{p_o}$ be the left translation $\mathbb{H} \ni p \mapsto p_o \circ p$. Then $L_{p_o}$ is smooth and its Jacobian matrix in $q=(x,y,t)$ is
\[
J_{p_o}(q)=
\begin{pmatrix}
1 & 0 & 0 \\
0 & 1 & 0 \\
-2y_o & 2x_o & 1 
\end{pmatrix};
\]
in particular $|\det(J_{p_o}(q))|=1$.

In the next proposition, we show that for $C^3$ functions the AMV Laplacian on $(\mathbb{H},\dist_{CC},\leb^3)$ coincides with the usual Kohn Laplacian $\Delta_{\mathbb{H}}=X^2 + Y^2$ up to a multiplicative constant. This result is already known from the work of Ferrari et al.~\cite{MR3177660}, but we provide a slightly different proof for the reader's convenience.

\begin{prop}\label{prop:Kohn}
Let $u \in C^3(\mathbb{H})$. Then
\[
\Delta_{\leb^3}^{\dist_{CC}}u=c\Delta_{\mathbb{H}}u,
\]
where $c=\frac{1}{2}\fint_{B_1(o)}x^2 \di \leb^3$.
\end{prop}

\begin{proof}
Let $u \in C^3(\mathbb{H})$ and $p_o=(x_o,y_o,t_o) \in \mathbb{H}$. Note that in this proof, all the balls are taken with respect to the Carnot--Carathéodory metric. For any given $r>0$, by the change of variable $p=L_{p_o}(q)$ we have
\begin{align}\label{eq:Heisenberg1}
\Delta_{\leb^3,r}^{\dist_{CC}}u(p_o) = \frac{1}{r^2}\fint_{B_r(p_o)} u(p)-u(p_o)\di p & = \frac{1}{r^2}\fint_{B_r(o)} [u(L_{p_o}(q))-u(p_o)] \left|\det(J_{p_o}(q))\right| \di q\nonumber\\
& = \frac{1}{r^2}\fint_{B_r(o)} [v(q)-v(o)] \di q,
\end{align}
where we have set $v = u\circ L_{p_o}$. Let us write the Taylor expansion of $v$ at $o$ (see \cite[p.~743]{MR2363343}\footnote{It is at this point we need $u\in C^3(\mathbb{H})$.}): for $q=(x,y,t) \in B_r(o)$,
\begin{align*}
v(q) & = v(o)+(X v)(o) x + (Y v)(o) y + (T v)(o) t\\
& + \frac{1}{2} (X^2 v)(o) x^2 + \frac{1}{2} (Y^2 v)(o) y^2 + [2(Tv)(o)+(XYv)(o)]xy + O(r^3).
\end{align*}
Since the vector fields $X$, $Y$ and $T$ commute with left translations,
\begin{align*}
v(q) & = v(o)+(X u)(p_o) x + (Y u)(p_o) y + (T u)(p_o) t\\
& + \frac{1}{2} (X^2 u)(p_o) x^2 + \frac{1}{2} (Y^2 u)(p_o) y^2 + [2(Tu)(p_o)+(XYu)(p_o)]xy + O(r^3).
\end{align*}
Plugging this expression into \eqref{eq:Heisenberg1} and using the fact that $\leb^3(B_r(p_o))=\leb^3(B_r(o))$, we obtain
\begin{align*}
\Delta_{\leb^3,r}^{\dist_{CC}}u(p_o) & = \frac{1}{r^2} \left( (X u)(p_o) \fint_{B_r(o)}x\di q + (Y u)(p_o) \fint_{B_r(o)}y\di q + (T u)(p_o) \fint_{B_r(o)}t\di q \right.\\
& \qquad \quad + \frac{1}{2} (X^2 u)(p_o) \fint_{B_r(o)}x^2\di q \,\, + \, \, \frac{1}{2} (Y^2 u)(p_o) \fint_{B_r(o)}y^2\di q  \\
& \left. \qquad \quad + \, \, [2(Tu)(p_o)+(XYu)(p_o)]\fint_{B_r(o)}xy\di q \right) \, \, + \, \, O(r).
\end{align*}
Now, it is known that a Carnot--Carathéodory ball centered at the origin of $\mathbb{H}$ is symmetric around the $t$-axis and also symmetric with respect to the $xy$-plane, hence
\[
\int_{B_r(o)}x\di q = \int_{B_r(o)}y\di q = \int_{B_r(o)}t\di q = \int_{B_r(o)}xy\di q = 0.
\]
Moreover, a Carnot--Carathéodory ball centered at the origin is invariant under rotations around the $z$-axis, consequently
\[
\int_{B_r(o)}x^2\di q = \int_{B_r(o)}y^2\di q.
\]
Therefore we get
\begin{align*}
\Delta_{\leb^3,r}^{\dist_{CC}}u(p_o) & = \frac{(X^2u)(p_o)+(Y^2u)(p_o)}{2 r^2} \fint_{B_r(o)} x^2 \di q + O(r)\\
& = \Delta_{\mathbb{H}}u(p_o) \frac{1}{2r^2} \fint_{B_r(o)} x^2 \di q + O(r)
\end{align*}
Since $B_r(o)=\delta_r(B_1(o))$ and the Jacobian determinant of $\delta_r$ is constant equal to $r^4$, the change of variable $q=\delta_r(q')$  provides $\leb^3(B_r(o))=r^4 \leb^3(B_1(o))$ and $\int_{B_r(o)}x^2 \di q=r^4 \int_{B_r(o)} (rx)^2 \di q$. Therefore
$$
\Delta_{\leb^3,r}^{\dist_{CC}}u(p_o) = \frac{\Delta_{\mathbb{H}}u(p_o)}{2} \fint_{B_1(o)} x^2 \di q + O(r),
$$
hence the result.
\end{proof}

\subsection{Weighted Lebesgue measures in $\mathbb{R}^n$}

In an earlier work by A.K.~Bose \cite{MR177128}, the author considers weighted Lebesgue measures
\[
\mu= w \leb^n
\]
on connected open sets $\Omega$ in $\setR^n$, where $w \in L^1_{\text{loc}}(\Omega,\leb^n)$ is nonnegative and such that $\mu(B)>0$  
for any ball $B \subset \Omega$. He shows that when $w \in C^1(\Omega)$, any function $u \in L^1_{\text{loc}}(\Omega,\mu)$ which satisfies the mean value property,
\[
u(x) = \fint_{B_r(x)} u \di \mu \qquad \forall B_r(x)\subset \Omega,
\]
is a $C^2$ function and a solution to the partial differential equation
\begin{equation}\label{eq:Bose1}
L_wu\df w \Delta u + 2 \nabla w \cdot \nabla u = 0
\end{equation}
in $\Omega$, hence a harmonic function for the weighted Laplacian $L_w$ (see \cite{kijowski} for the case of a Sobolev regular weight function). The converse is not true \cite{MR177128}: consider the example $w(x,y)=(x+y)^2$ and $u(x,y)=x^2-3xy+y^2$ in $\setR^2$. A direct computation shows that $u$ satisfies \eqref{eq:Bose1}, but
\[
\fint_{B_r(x,y)}u \di \mu = u(x,y) + \frac{r^4}{6(r^2+2(x+y)^2)} \neq u(x,y)
\]
for all $(x,y) \in \setR^2$ and $r>0$. However, with our notation, we have
\[
\Delta_{\mu,r}^{\dist_e} u(x,y) = \frac{r^2}{6(r^2+2(x+y)^2)}
\]
for any $(x,y) \in \setR^2$ and $r>0$, letting $r$ tend to $0$ shows that $u$ is AMV harmonic outside the diagonal $\{x=-y\}$, which coincides with $\{w=0\}$. More generally, we have the following.

\begin{prop}
Let $w \in C^1(\Omega)$ and $u \in C^2(\Omega)$. Then for any $x \in \Omega$ such that $w(x) \neq 0$,
\begin{equation}\label{eq:Bose1.1}
\Delta_{\mu}^{\dist_e}u(x) =c_n
\frac{L_w u (x)}{w(x)}\,,
\end{equation}
where $c_n=2^{-1}(n+2)^{-1}$. Moreover, for any $x \in \Omega$ such that $w(x)=0$, assume that $
b^r(x)\df\frac{1}{r^2}\fint_{B_r(x)} (y-x) w(y)\di y$ and
$ a^r_{ij}(x) \df \frac{1}{2r^2} \fint_{B_r(x)} (y-x)_i (y-x)_j w(y) \di y$ converges  to $b(x) \in \setR^n$ and $a_{ij}(x) \in \setR$ respectively when $r \to 0^+$. Then
\begin{equation}\label{eq:Bose1.2}
\Delta_{\mu}^{\dist_e}u(x) =
\sum_{i,j=1}^n a_{ij}(x)\partial_{ij}u(x) + b(x) \cdot \nabla u(x).
\end{equation}
\end{prop}

\begin{proof}
Let $x \in \Omega$ be such that $w(x)\neq 0$. By first and second order Taylor expansions of $w$ and $u$ respectively, we know that there exists a function $E:(0,+\infty) \to (0,+\infty)$ such that $E(r) \to 0$ when $r \to 0^+$ and for any $r>0$ and $y \in B_r(x)$,
\begin{align*}
(u(y)-u(x))w(y) & =w(x)\nabla u (x) \cdot (y-x)+ \frac{w(x)}{2}\nabla^2 u(x) \cdot(y-x,y-x) \\
& + [\nabla u (x) \cdot (y-x)][\nabla w (x) \cdot (y-x)] + E(r) r^2.
\end{align*}
Note that for any $v \in \setR^n$, the antisymmetry of $y \mapsto v \cdot (y-x)$ with respect to the hyperplane $v^{\perp}$ implies that $\int_{B_r(x)}v \cdot (y-x) \di x=0$, and that $\int_{B_r(x)}(y-x)_i(y-x)_j \di y= 0$ for $i\neq j$ and $\int_{B_r(x)}(y-x)_i^2\di y= 2 c_n \leb^n(B_r(x))r^2$ for any $1\le i,j \le n$. From this, a direct computation shows that 
\[
\Delta_{\mu,r}^{\dist} u(x)=c_n \frac{\leb^n(B_r(x))}{\mu(B_r(x))}(w(x)\Delta u(x)+2 \nabla u(x) \cdot \nabla w(x) + o(1)).
\]
Since $w$ is $C^1$, $\mu(B_r(x))/\leb^n(B_r(x))\to w(x)$ as $r \to 0^+$, hence \eqref{eq:Bose1.1}.

Now take $x \in \Omega$ such that $w(x) \neq 0$ and the required assumption is satisfied. By a similar expansion as above, but with respect to $u$ only gives that
\[
\Delta_{\mu,r}^{\dist_e} u(x)= \nabla u(x) \cdot b^r(x) + \sum_{i,j=1}^n \partial_{ij}u(x) a_{ij}^r + o(1),
\]
hence \eqref{eq:Bose1.2}.
\end{proof}

Note that \eqref{eq:Bose1.1} is consistent with the unweighted case $w\equiv 1$. Moreover, in the example $w(x,y)=
(x+y)^2$, explicit calculations from \eqref{eq:Bose1.2} show that for any $u \in C^2(\setR^2)$,
\[
\Delta_{\mu}^{\dist_e}u = \frac{1}{6}(\Delta u  + \partial_{xy}u ) \qquad \text{on $\{x=-y\}$}.
\]
In particular, for $u(x,y)=x^2-3xy+y^2$ we get $\Delta_\mu^{\dist_e}u=1/6$ on $\{x=-y\}$, hence $u$ is not AMV harmonic.

\subsection{The Lebesgue measure with a Dirac mass in $\mathbb{R}^n$}\label{euclideandirac}


Let us consider $(X,\dist,\mu)\df (\setR^n,\dist_e,\leb^n+\delta_o)$ where $\delta_o$ is a Dirac measure at the origin $o$ of $\setR^n$. Take $u \in L^1_{\text{loc}}(\setR^n,\leb^n)$ and $x \in \setR^n$. If $x\neq o$, we trivially get $\Delta_\mu^{\dist_e}u(x) = \Delta_{\leb^n}^{\dist_e}u(x)$. Consider therefore the case $x =o$. Since $\mu(B_r(o))=1+\omega_n r^n$ (where $\omega_n = \leb^n(B_1)$) for any $r>0$,
\begin{align*}
\Delta_{\mu,r}^{\dist_e}u(o) & = \frac{1}{r^2(1+\omega_n r^n)}\Bigg( \int_{B_r(o)}u(y)-u(x)\di y + \underbrace{\int_{B_r(o)}u(y)-u(o)\di \delta_o(y)}_{=u(o)-u(o)=0} \Bigg)\\
& = \frac{\leb^n(B_r(o))}{r^2(1+\omega_n r^n)}\fint_{B_r(o)} u(y) - u(o) \di y.
\end{align*}
Since $1/(1+\omega_n r^n) = 1 + O(r^n)$, we get
\begin{equation}\label{eq:Leb}
\Delta_{\mu,r}^{\dist_e}u(o) = \omega_n r^{n-2} (1+O(r^n)) \fint_{B_r(o)} u(y) - u(o) \di y.
\end{equation}

This simple computation leads to the following:

\begin{observation}
Assume that $o$ is a Lebesgue point of $u$ with respect to $\leb^n$ and denote by $u^*(o)$ the unique number $a \in \setR$ such that $\fint_{B_r (x)} |u(y)-a|\di y \to 0$ when $r \to 0^+$. If $n=1$, assume additionally that  $r^{-1}\fint_{B_r(o)} u(y) - u(o) \di y$ converges to some constant $b$ as $r \to 0^+$. Then:
\begin{equation}\label{eq:Leb2}
\Delta_{\mu}^{\dist_e}u(o) = 
\begin{cases}
0 &  \text{when $n \ge 3$,}\\ \pi(u^*(o)-u(o)) & \text{when $n =2$},\\
2b & \text{when $n=1$}. 
\end{cases}
\end{equation}
\end{observation}

Note that in case $n=1$, we obviously have $u^*(o)=u(o)$. Moreover, $u$ differentiable in $o$ is enough to imply convergence of $\Delta_{\mu,r}^{\dist_e}u(o)$ towards $0$ as $r \to 0^+$.


It is also worth mentioning that for the Poisson problem $\Delta_{\mu}^{\dist_e}u = v$, when $n\ge3$, a necessary condition for existence of a solution is $v(o)=0$.

This last example is a special case of a stratified measure that we discuss in the next section.

\section{AMV Laplacian for Stratified Measures}\label{sec:stratified}

In this section, we introduce the notion of a stratified measure and study the AMV Laplacian at the intersection of strata of such a measure. We then apply our results to the case of intersecting submanifolds in $\mathbb{R}^n$.

Recall that for $Q\ge 0$, a Borel measure $\meas$ on a metric space $(X,\dist)$ is called Ahlfors $Q$-regular if it satisfies $cr^{Q}\le\meas(B)\le Cr^{Q}$ for any metric ball $B \subset X$ with radius $r>0$, where $c,C>0$ are independent of the radius and the ball. If a measure $\mu$ is $Q$-Ahlfors regular, we say that $\mu$ has an Ahlfors dimension equal to $Q$.

\begin{defn}[Stratified Measures]
Let $(X,\dist)$ be a metric space. We call $\meas$ a \emph{stratified measure} on $(X,\dist)$ if
\[
\mu = \mu_1 + \cdots + \mu_k,
\]
where for any $1 \le j \le k$,
\begin{enumerate}[label=(\roman*)]
	\item $\mu_j$ is a measure supported on a closed set $Y_j \subset X$,
	\item $\mu_j$ is $Q_j$-Ahlfors regular on $(Y_j,\dist_{|Y_j})$,
	\item $Q_1<\ldots<Q_k$.
\end{enumerate}

\end{defn}

A particular example of a stratified measure is $\mathcal{H}^{m_{1}}\measrestr M_{1}+\ldots+\mathcal{H}^{m_{k}}\measrestr M_{k}$ where $M_1^{m_1}, \ldots, M_k^{m_k}$ are submanifolds of $\setR^n$ such that $m_1<\ldots<m_k$. See also Example \ref{ex:example???}.

The next theorem states that the lowest dimensional stratum determines the AMV Laplacian of a stratified measure.

\begin{thm}
	[AMV Laplacian on Intersections of Strata of a Stratified Measure]\label{thm:intersection-ahlfors-regular-measures}
	Let $(X,\dist)$ be a metric space equipped with a stratified measure $\mu$ and $u \in L^1_{\text{loc}}(X,\mu)$. For any $x\in\bigcap_{i=1}^{l}Y_{j_{i}}$ where $\{j_{i}\}$ is an increasing subsequence	of $\{1,\ldots,k\}$, if $\Delta_{\meas_{j_{1}}}^{\dist}u(x)$ exists and $r^{Q_{j_{i}}-Q_{j_{1}}}|\Delta_{\mu_{j_{i}},r}^{\dist}u(x)|\to 0$ as $r\to0^{+}$ for all $2\le i\le l$, then $\Delta_{\meas}^{\dist}u(x)$ exists and
	\[
	\Delta_{\meas}^{\dist}u(x)=\Delta_{\meas_{j_{1}}}^{\dist}u(x).
	\]
\end{thm}

\begin{proof}
	We can without loss of generality assume that $\{j_{i}\}=\{j\}=\{1,\ldots,l\}$: indeed, since $Y_{j}^{c}$ is open for any $1\le j\le k$, there exists $r>0$ small enough such that $\mu \measrestr B_r(x) = \mu_{j_{1}} \measrestr B_r(x) + \ldots + \mu_{j_{l}} \measrestr B_r(x)$, hence the validity of the above relabeling. Let therefore $x\in\bigcap_{j=1}^{l}Y_{j}$
	and consider $\Delta_{\meas,r}^{\dist}u(x)$. Then
	\begin{align*}
	\Delta_{\meas,r}^{\dist}u(x) & =\frac{1}{r^{2}\meas(B_{r}(x))}\int_{B_{r}(x)}u(y)-u(x)\di\meas(y)\\
	& =\frac{1}{r^{2}\mu(B_{r}(x))}\sum_{j=1}^{l}\int_{B_{r}(x)\cap Y_{j}}u(y)-u(x)\di\meas_{j}(y).
	\end{align*}
	For a given $j$, let $c_j,C_j$ be the Ahlfors constants related to $\mu_j$.  	Note that 
\[
	\meas(B_{r}(x)) =\sum_{j=1}^{l}\mu_{j}(B_{r}(x)\cap Y_{j})
	\ge\sum_{j=1}^{l}c_{j}r^{Q_{j}}\ge c_{1}r^{Q_{1}}.
\]
	For $j\ge2$ we have, by assumption,
	\begin{align}
	& \phantom{{}={}} \bigg|\frac{1}{r^{2}\meas(B_{r}(x))}\int_{B_{r}(x)\cap Y_{j}}u(y)-u(x)\di\meas_{j}(y)\bigg|\nonumber \\
	& \le \frac{1}{c_{1}r^{Q_{1}+2}}\bigg|\int_{B_{r}(x)\cap Y_{j}}u(y)-u(x)\di\meas_{j}(y)\bigg|\nonumber \\
	& =  \frac{\meas_{j}(B_{r}(x)\cap Y_{j})}{c_{1}r^{Q_{1}+2}}\bigg|\fint_{B_{r}(x)\cap Y_{j}}u(y)-u(x)\di\meas_{j}(y)\bigg|\nonumber \\
	& \le  \frac{C_{j}}{c_{1}}r^{Q_{j}-Q_{1}}\bigg|\Delta_{\mu_{j},r}^{\dist}u(x)\bigg|\to0,\qquad r\to0^{+}.\label{eq:lowerordermeasures}
	\end{align}
	Then
	\begin{align*}
	& \phantom{{}={}} \left|\Delta_{\meas,r}^{\dist}u(x)-\Delta_{\meas_{1}}^{\dist}u(x)\right|\\
	& =\bigg|\frac{1}{r^{2}\meas(B_{r}(x))}\sum_{j=1}^{l}\int_{B_{r}(x)\cap Y_{j}}u(y)-u(x)\di\meas_{j}(y)-\Delta_{\meas_{1}}^{\dist}u(x)\bigg|\\
	& =\bigg|\sum_{j=1}^{l}\frac{\meas_j(B_{r}(x)\cap Y_{j})}{r^{2}\meas(B_{r}(x))}\fint_{B_{r}(x)\cap Y_{j}}u(y)-u(x)\di\meas_{j}(y)-\Delta_{\meas_{1}}^{\dist}u(x)\bigg|\\
	& =\bigg|\sum_{j=1}^{l}\frac{\meas_j(B_{r}(x)\cap Y_{j})}{\meas(B_{r}(x))}\Delta_{\meas_{j},r}^{\dist}u(x)-\Delta_{\meas_{1}}^{\dist}u(x)\bigg|\\
	& \le\bigg|\frac{\meas_1(B_{r}(x)\cap Y_{1})}{\meas(B_{r}(x))}\Delta_{\meas_{1},r}^{\dist}u(x)-\Delta_{\meas_{1}}^{\dist}u(x)\bigg|+\bigg|\sum_{j=2}^{l}\frac{\meas_j(B_{r}(x)\cap Y_{j})}{\meas(B_{r}(x))}\Delta_{\meas_{j},r}^{\dist}u(x)\bigg| .
	\end{align*}
	The second term tends to zero as $r\to0^{+}$ by \eqref{eq:lowerordermeasures}.
	Moreover,
	\begin{align*}
	\frac{\meas_1(B_{r}(x)\cap Y_{1})}{\meas(B_{r}(x))} & =\frac{\meas_1(B_{r}(x)\cap Y_{1})}{\sum_{j=1}^{k}\mu_{j}(B_{r}(x)\cap Y_{j})}\le1,\\
	\frac{\meas_1(B_{r}(x)\cap Y_{1})}{\meas(B_{r}(x))} & =\frac{\meas_1(B_{r}(x)\cap Y_{1})}{\sum_{j=1}^{k}\mu_{j}(B_{r}(x)\cap Y_{j})}\\
	& =\frac{1}{1+\sum_{j=2}^{k}\frac{\mu_{j}(B_{r}(x)\cap Y_{j})}{\mu_{1}(B_{r}(x)\cap Y_{1})}}\\
	& \ge\frac{1}{1+\sum_{j=2}^{k}\frac{C_{j}}{c_{1}}r^{Q_{j}-Q_{1}}}\to1,\qquad r\to0^{+}.
	\end{align*}
	Therefore
	\[
	\limsup_{r\to0^{+}}\left|\Delta_{\meas,r}^{\dist}u(x)-\Delta_{\meas_{1}}^{\dist}u(x)\right|
	 =\limsup_{r\to0^{+}}\left|\frac{\meas_1(B_{r}(x)\cap Y_{j})}{\meas(B_{r}(x))}\Delta_{\meas_{1},r}^{\dist}u(x)-\Delta_{\meas_{1}}^{\dist}u(x)\right|
	 =0.
\]
\end{proof}
\begin{rem}
	Note that a sufficient condition for $r^{Q_{j_{i}}-Q_{j_{1}}}|\Delta_{\mu_{j_{i}},r}^{\dist}u(x)|$
	to be $o(1)$ is boundedness of $|\Delta_{\mu_{j_{i}},r}^{\dist}u(x)|$
	for all $2\le i\le l$.
\end{rem}

The previous theorem can be extended in a straightforward way to sums
of nondecreasing Ahlfors regular measures if limits of the
form $\frac{\meas_j(B_{r}(x)\cap Y_{j})}{\meas(B_{r}(x))}$ exist as
$r$ tends to $0^{+}$ for the lowest dimensional stratum, reducing the AMV Laplacian to a convex
combination of the AMV Laplacians of this stratum.

\begin{cor}
	\label{cor:intersection-ahlfors-regular-measures-nondecreasing} Let
	$(X,\dist,\meas)$ and $x\in\bigcap_{i=1}^{l}Y_{j_{i}}$ be as in
	Theorem \ref{thm:intersection-ahlfors-regular-measures} except that
	$Q_{j}$ is nondecreasing with $Q_{j_{1}}=Q_{j_{2}}=\ldots=Q_{j_{n}}<Q_{j_{n+1}}$ and
	$r^{Q_{j_{i}}-Q_{j_{1}}}|\Delta_{\mu_{j_{i}},r}^{\dist}u(x)|\to 0$
	as $r\to0^{+}$ for any $n+1\le i\le l$. Assume that $\Delta_{\meas_{j_{1}}}^{\dist}u(x)$,
	$\Delta_{\meas_{j_{2}}}^{\dist}u(x),\ldots,\Delta_{\meas_{j_{n}}}^{\dist}u(x)$
	exist. Assume also that $\alpha_{i}(x) \df \lim_{r\to0^{+}}\meas_j(B_{r}(x)\cap Y_{j_{i}})/\meas(B_{r}(x))$
	exists for any $1\le i\le n$. Then $\Delta_{\meas}^{\dist}u(x)$ exists,
	\[
	\Delta_{\meas}^{\dist}u(x)=\sum_{i=1}^{n}\alpha_{i}(x)\Delta_{\meas_{j_{i}}}^{\dist}u(x),
	\]
	and $\sum_{i=1}^{n}\alpha_{i}(x)=1$.
\end{cor}

\begin{proof}
	As in the proof of Theorem \ref{thm:intersection-ahlfors-regular-measures}, we have $\meas(B_{r}(x)) \ge c_{1}r^{Q_{1}}$, and we can reduce to the case $\{j_{i}\}=\{j\}=\{1,\ldots,n,\ldots,l\}$ for which
	\[
	\Delta_{\meas,r}^{\dist}u(x)=\frac{1}{r^{2}\meas(B_{r}(x))}\sum_{j=1}^{l}\int_{B_{r}(x)\cap Y_{j}}u(y)-u(x)\di\meas_{j}(y).
	\]
	
	For $j\ge n+1$ we have, by assumption,
	\begin{align}
	& \phantom{{}={}} \bigg|\frac{1}{r^{2}\meas(B_{r}(x))}\int_{B_{r}(x)\cap Y_{j}}u(y)-u(x)\di\meas_{j}(y)\bigg|\nonumber \\
	& \le  \frac{1}{c_{1}r^{Q_{1}+2}}\bigg|\int_{B_{r}(x)\cap Y_{j}}u(y)-u(x)\di\meas_{j}(y)\bigg|\nonumber \\
	& =  \frac{\meas_{j}(B_{r}(x)\cap Y_{j})}{c_{1}r^{Q_{1}+2}}\bigg|\fint_{B_{r}(x)\cap Y_{j}}u(y)-u(x)\di\meas_{j}(y)\bigg|\nonumber \\
	& \le \frac{C_{j}}{c_{1}}r^{Q_{j}-Q_{1}}\left|\Delta_{\mu_{j},r}^{\dist}u(x)\right|\to0,\qquad r\to0^{+}\label{eq:lowerordermeasures-nondecreasing}.
	\end{align}
	Continuing,
	\begin{align*}
	& \phantom{{}={}} \bigg|\Delta_{\meas,r}^{\dist}u(x)-\sum_{j=1}^{n}\alpha_{j}(x)\Delta_{\meas_{j}}^{\dist}u(x)\bigg|\\
	& =\bigg|\frac{1}{r^{2}\meas(B_{r}(x))}\sum_{j=1}^{l}\int_{B_{r}(x)\cap Y_{j}}u(y)-u(x)\di\meas_{j}(y)-\sum_{j=1}^{n}\alpha_{j}(x)\Delta_{\meas_{j}}^{\dist}u(x)\bigg|\\
	& =\bigg|\sum_{j=1}^{l}\frac{\meas_j(B_{r}(x)\cap Y_{j})}{r^{2}\meas(B_{r}(x))}\fint_{B_{r}(x)\cap Y_{j}}u(y)-u(x)\di\meas_{j}(y)-\sum_{j=1}^{n}\alpha_{j}(x)\Delta_{\meas_{j}}^{\dist}u(x)\bigg|\\
	& =\bigg|\sum_{j=1}^{l}\frac{\meas_j(B_{r}(x)\cap Y_{j})}{\meas(B_{r}(x))}\Delta_{\meas_{j},r}^{\dist}u(x)-\sum_{j=1}^{n}\alpha_{j}(x)\Delta_{\meas_{j}}^{\dist}u(x)\bigg|\\
	& \le\bigg|\sum_{j=1}^{n}\frac{\meas_j(B_{r}(x)\cap Y_{j})}{\meas(B_{r}(x))}\Delta_{\meas_{j},r}^{\dist}u(x)-\sum_{j=1}^{n}\alpha_{j}(x)\Delta_{\meas_{j}}^{\dist}u(x)\bigg|\\
	& \qquad+\bigg|\sum_{j=n+1}^{l}\frac{\meas_j(B_{r}(x)\cap Y_{j})}{\meas(B_{r}(x))}\Delta_{\meas_{j},r}^{\dist}u(x)\bigg|
	\end{align*}
	The second term tends to zero as $r\to0^{+}$ by \eqref{eq:lowerordermeasures-nondecreasing}.
	Also, by assumption,
	\begin{align*}
	\frac{\meas_j(B_{r}(x)\cap Y_{j})}{\meas(B_{r}(x))} & \to\alpha_{j}(x),\qquad r\to0^{+}.
	\end{align*}
	Therefore,
	\begin{align*}
	& \phantom{{}={}} \limsup_{r\to0^{+}}\left|\Delta_{\meas,r}^{\dist}u(x)-\sum_{j=1}^{n}\alpha_{j}(x)\Delta_{\meas_{j}}^{\dist}u(x)\right|\\
	& \le \limsup_{r\to0^{+}}\left|\sum_{j=1}^{n}\left(\frac{\meas_j(B_{r}(x)\cap Y_{j})}{\meas(B_{r}(x))}\Delta_{\meas_{j},r}^{\dist}u(x)-\alpha_{j}(x)\Delta_{\meas_{j}}^{\dist}u(x)\right)\right| =0.
	\end{align*}
	The fact $\sum_{j=1}^{n}\alpha_{j}(x)=1$ is immediate since
	\begin{align*}
	\bigg|\sum_{j=1}^{n}\alpha_{i}(x)-1\bigg| & =\bigg|\sum_{j=1}^{n}\lim_{r\to0^{+}}\frac{\meas_j(B_{r}(x)\cap Y_{j})}{\meas(B_{r}(x))}-1\bigg|\\
	& =\bigg|\lim_{r\to0^{+}}\frac{1}{\sum_{j=1}^{l}\meas(B_{r}(x)\cap Y_{j})}\sum_{j=1}^{n}\meas_j(B_{r}(x)\cap Y_{j})-1\bigg|\\
	& =\bigg|\lim_{r\to0^{+}}\frac{1}{\sum_{j=1}^{l}\meas(B_{r}(x)\cap Y_{j})}\bigg(\sum_{j=1}^{n}\meas_j(B_{r}(x)\cap Y_{j})-\sum_{j=1}^{l}\meas_j(B_{r}(x)\cap Y_{j})\bigg)\bigg|\\
	& =\bigg|\lim_{r\to0^{+}}\frac{1}{O(r^{Q_{1}})}O(r^{Q_{n+1}})\bigg|\\
	& =0.
	\end{align*}
\end{proof}

\begin{example}
An interesting example is $(\setR^2,\dist_e,\meas\df\mu_1 + \ldots + \mu_l)$ where $\mu_i=\haus^1 \measrestr S_i$ for any $1 \le i \le l$ and $S_i$ is the image of a smooth curve $c_i$ emanating from $o=(0,0)$ with direction $\tau_i$. Here $\haus^1$ denotes the $1$-dimensional Hausdorff measure. A direct computation shows that for any $u \in C^2(\setR^2)$,
$$
\Delta_{\meas}^{\dist_e}u(o) \in \setR \quad \iff \quad \sum_{i=1}^l \partial_{\tau_i} u(o) = 0,
$$
where $\partial_{\tau_i}u$ is the directional derivative of $u$ along $\tau_i$. This is the well-known Kirchhoff condition (compare for instance with \cite[Section 4]{MR3985395}). If this condition holds, Corollary \ref{cor:intersection-ahlfors-regular-measures-nondecreasing} implies
$$
\Delta_{\meas}^{\dist_e}u(o) = \frac{1}{l} \sum_{i=1}^l \Delta_{\meas_i}^{\dist_e}u(o).
$$
\end{example}

Let us apply the previous  results to the case of intersecting submanifolds in $\setR^n$. Recall that a Riemannian submanifold of $\setR^n$ is equipped with the Riemannian metric inherited from the Euclidean metric of $\setR^n$. If $(M,g)$ is a smooth $m$-dimensional Riemannian submanifold of $\setR^n$, then the topological metric $\dist_g$ induced by $g$ in the usual way (i.e.~minimizing the length of curves joining two points) satisfies $\dist_g \ge \dist_e$, and the canonically associated Riemannian volume measure $\vol_g$ on $M$ coincides with the $m$-dimensional Hausdorff measure $\haus^m$. 

\begin{prop}\label{prop:MVL-equals-Delta-g}
Let $(M,g)$ be a smooth $m$-dimensional Riemannian submanifold of $\setR^n$, and consider $(\setR^n,\dist_e,\haus^m \measrestr M$). Then, for any $u\in C^{2}(\setR^n)$ and any interior point $x$ in $M$,
	\[
	\Delta_{\mathcal{H}^{m}\measrestr M}^{\dist_{e}}u(x)=\frac{\Delta_{g}u(x)}{2(m+2)},
	\]
	where $\Delta_{g}$ is the Laplace-Beltrami operator on $(M,g)$.
\end{prop}

\begin{proof}
	Since $\dist_g\ge \dist_e$, the geodesic ball $B_{r}^{g}(x)$ is included in the subset $B_{r}(x)\cap M$ of $\setR^n$. Then
	\begin{align*}
	\Delta_{\mathcal{H}^{m}\measrestr M,r}^{\dist_{e}}u(x) & =\frac{1}{r^{2}}\fint_{B_{r}(x)\cap M}u(y)-u(x)\di\mathcal{H}^{m}(y)\\
	& =\frac{1}{r^{2}\mathcal{H}^{m}(B_{r}(x)\cap M)}\bigg(\int_{(B_{r}(x)\cap M)\backslash B_{r}^{g}(x)}u(y)-u(x)\di\mathcal{H}^{m}(y)\\
	& \qquad+\int_{B_{r}^{g}(x)}u(y)-u(x)\di V_{g}(y)\bigg).
	\end{align*}
	For the first term, by a first order Taylor expansion of $u$ around
	$x$,
	\begin{align*}
	& \left|\frac{1}{r^{2}\mathcal{H}^{m}(B_{r}(x)\cap M)}\int_{(B_{r}(x)\cap M)\backslash B_{r}^{g}(x)}u(y)-u(x)\di\mathcal{H}^{m}(y)\right|\\
	\le & \frac{O(r)}{r^{2}\mathcal{H}^{m}(B_{r}(x)\cap M)}(\mathcal{H}^{m}(B_{r}(x)\cap M)-\mathcal{H}^{m}(B_{r}^{g}(x))).
	\end{align*}
	Now, from the works of Karp and Pinsky \cite{MR967797}, the volume
	of the extrinsic ball for small $r$ is given by
	\begin{equation}
	\mathcal{H}^{m}(B_{r}(x)\cap M)=\omega_{m}r^{m}\left(1+\frac{2\lVert II(x)\rVert-\lVert H(x)\rVert}{8(m+2)}r^{2}+O(r^{3})\right),\label{eq:volume-extrinsic-ball}
	\end{equation}
	where $II$ denotes the second fundamental form of $(M,g)$ and $H=\Tr II$
	the mean curvature. Furthermore, the volume of the intrinsic ball
	has been calculated by Gray \cite{MR0339002} for small $r$ as
	\[
	(\haus^m(B_{r}^{g}(x))\, =\, )\, \, \, \vol_g(B_{r}^{g}(x))=\omega_{m}r^{m}\left(1-\frac{R(x)}{6(m+2)}r^{2}+O(r^{4})\right),
	\]
	where $R$ is the scalar curvature on $(M,g)$. Hence
	\[
	\mathcal{H}^{m}(B_{r}(x)\cap M)-\mathcal{H}^{m}(B_{r}^{g}(x))=O(r^{m+2}).
	\]
	Therefore 
	$$
	\frac{O(r)}{r^{2}\mathcal{H}^{m}(B_{r}(x)\cap M)}(\mathcal{H}^{m}(B_{r}(x)\cap M)-\mathcal{H}^{m}(B_{r}^{g}(x)))  =\frac{O(r)O(r^{m+2})}{O(r^{m+2})}=O(r),
	$$
	so the first term tends to zero with $r$. The second term tends to
	$\frac{\Delta_{g}u(x)}{2(m+2)}$ since 
	\[
	\frac{\mathcal{H}^{m}(B_{r}^{g}(x))}{\mathcal{H}^{m}(B_{r}(x)\cap M)}=\frac{\omega_{m}r^{m}(1+O(r^{2}))}{\omega_{m}r^{m}(1+O(r^{2}))}\to1,\qquad r\to0^{+},
	\]
	and by the works of Gray and Willmore \cite{MR677493},
	\[
	\frac{1}{r^{2}}\fint_{B_{r}^{g}(x)}u(y)-u(x)\di V_{g}\to\frac{\Delta_{g}u(x)}{2(m+2)}\, \cdot
	\]
	Consequently,
	\begin{align*}
	\frac{1}{r^{2}\mathcal{H}^{m}(B_{r}(x)\cap M)}\int_{B_{r}^{g}(x)}u(y)-u(x)\di V_{g} & =\frac{\mathcal{H}^{m}(B_{r}^{g}(x))}{r^{2}\mathcal{H}^{m}(B_{r}(x)\cap M)}\fint_{B_{r}^{g}(x)}u(y)-u(x)\di V_{g}\\
	& \to\frac{\Delta_{g}u(x)}{2(m+2)}\, \cdot
	\end{align*}
\end{proof}

In this smooth context, we can prove the following refinement of Corollary \ref{cor:intersection-ahlfors-regular-measures-nondecreasing} where we get a mean value of the AMV Laplacians of the lowest stratum.

\begin{cor}\label{cor:???}
	Let $(X,\dist,\meas)=(\mathbb{R}^{n},\dist_{e},\mathcal{H}^{m_{1}}\measrestr M_{1}+\ldots+\mathcal{H}^{m_{k}}\measrestr M_{k})$
	where $\{m_j\}_j$ is a non-decreasing sequence of integers and $(M_j,g_j)$ is a smooth $m_j$-dimensional Riemannian submanifold of $\setR^n$ without boundary for any $1\le j \le k$. Take $u\in L_{\text{loc}}^{1}(X,\mu)$ and let $x\in\bigcap_{i=1}^{l}M_{j_{i}}$ for a subsequence $\{j_{i}\}$
	of $\{1,\ldots,k\}$ with $m_{j_{1}}=\ldots=m_{j_{t}}$.
If $\Delta_{\mathcal{H}^{m_{j_i}}\measrestr M_{1}}^{\dist}u(x)$ exist for all $i \in\{1,\ldots,t\}$ and $r^{m_{j_{i}}-m_{j_{1}}}|\Delta_{\mu_{j_{i}},r}^{\dist}u(x)|\to 0$
	as $r\to0^{+}$ for all $i \in\{t+1,\ldots,k\}$, then $\Delta_{\meas}^{\dist}u(x)$ exists and
	\[
	\Delta_{\meas}^{\dist}u(x)=\frac{1}{t}\sum_{i=1}^{t}\Delta_{\mathcal{H}^{m_{j_{i}}}\measrestr M_{j_{i}}}^{\dist}u(x).
	\]
	Furthermore, if $u\in C^{2}(\mathbb{R}^{n})$, we have
	\[
	\Delta_{\meas}^{\dist}u(x)=\frac{1}{2(m_{j_{1}}+2)t}\sum_{i=1}^{t}\Delta_{g_{i}}u(x).
	\]
\end{cor}

\begin{proof}
	The first part follows from Corollary \ref{cor:intersection-ahlfors-regular-measures-nondecreasing}
	if we can show that the limits $\alpha_{i}(x)\df\lim_{r\to0^{+}}\frac{\meas_{j_i}(B_{r}(x)\cap M_{j_{i}})}{\meas(B_{r}(x))}$ for $1\le i \le t$
	exist and are all equal to $\frac{1}{t}$. This is true since for any such $i$, applying
	\eqref{eq:volume-extrinsic-ball} with $m\df m_{j_1}$ gives
	\begin{align*}
	\lim_{r\to0^{+}}\frac{\meas_{j_i}(B_{r}(x)\cap M_{j_{i}})}{\meas(B_{r}(x))} & =\lim_{r\to0^{+}}\frac{\omega_{m}r^{m}(1+O(r^{2}))}{\sum_{k=1}^{l}\omega_{m}r^{m}(1+O(r^{2}))}\\
	& =\lim_{r\to0^{+}}\frac{\omega_{m}r^{m}(1+O(r^{2}))}{\sum_{k=1}^{t}\omega_{m}r^{m}(1+O(r^{2}))+O(r^{m+1})}\\
	& =\lim_{r\to0^{+}}\frac{\omega_{m}r^{m}(1+O(r^{2}))}{t\omega_{m}r^{m}(1+O(r^2))}=\frac{1}{t} \, \cdot
	\end{align*}
	For $u\in C^{2}(\mathbb{R}^{n})$,
	\[
	\frac{1}{t}\sum_{i=1}^{t}\Delta_{\mathcal{H}^{m_{j_{i}}}\measrestr M_{j_{i}}}^{\dist}u(x)=\frac{1}{2(m+2)t}\sum_{i=1}^{t}\Delta_{g_{i}}u(x)
	\]
	by Proposition \ref{prop:MVL-equals-Delta-g}.
\end{proof}

\begin{rem}
Let us point out that Theorem \ref{thm:intersection-ahlfors-regular-measures} and Corollary \ref{cor:intersection-ahlfors-regular-measures-nondecreasing} also hold true if the Ahlfors regularity assumption on the measures $\mu_i$ is replaced by a pointwise version, namely
$$
c r^{Q_{j_i}(x)} \le \mu_{j_i}(B_r(x)) \le Cr^{Q_{j_i}(x)}
$$
for any $r>0$, any $i$ and any $x \in \cap_{i} C_{j_i}$, where $\{j_i\}_i \subseteq \{1,\ldots,k\}$ and the constants $c,C>0$ might depend on $j_i$ and $x$.
\end{rem}

Let us apply this remark in the following example.

\begin{example}\label{ex:example???}
Set $L\df[0,1]\times \{0\}$, $S\df[-1,0]\times [-1/2,1/2]$, and consider $(\setR^2,\dist_{e})$ equipped with the measures $\mu^{1}, \mu^{2}, \mu^{3}$, where $\mu^{i}\df\mu_1^{i} + \mu_2$ and $$\mu_2\df\di x \di y \measrestr S, \qquad  \mu_1^{i}\df x^{i-1}\di x \measrestr L$$ for any $i \in \{1,2,3\}$. Then $\mu_2$ is $2$-Ahlfors regular, while at the intersection point $o=(0,0)$, for any $r>0$, one has
$$
\mu_1^{1}(B_r(o))=r, \qquad \mu_1^{2}(B_r(o))=\frac{r^2}{2}, \qquad \mu_1^{3}(B_r(o))=\frac{r^3}{3}\, \cdot
$$

Let us focus on $\mu^{1}$. An immediate computation shows that $\mu^{1}(B_r(o))=(r+\pi r^2)/2$ holds for any $r>0$ small enough. Take $u \in C^2(\setR^2)$. The second order Taylor expansion with Laplace remainder of $u(\cdot,0)$ at $0$
\, implies
$$
\Delta_{\mu_1^{1},r}^{\dist}u(o) = \frac{2}{r^3(1+\pi r)} \left( \frac{r^2}{2} \partial_x u (o) + \int_0^r \partial_{xx}^2u(t) \frac{(r-t)^2}{2} \di t \right).
$$
On the other hand, applying Taylor's theorem to $u(\cdot, \cdot)$ at $o$, we know that for some $d>0$ and $C>0$,
$$
|u(x,y)-u(o)-Du(o)\cdot(x,y)|\le C\|(x,y)\|^2 
$$
holds for all $(x,y) \in B_d(o)$. Therefore, for any $0<r<d$, we get
\begin{align*}
|\Delta_{\mu_2,r}^{\dist}u(o)| & \le \frac{2}{r^3(1+\pi r)} \left( \iint_{S \cap B_r(o)} |Du(o)|\|(x,y)\|\di x \di y + C \iint_{S \cap B_r(o)}\|(x,y)\|^2 \di x \di y \right)\\
& \le \frac{2}{r^3(1+\pi r)}\left( r |Du(o)| + Cr^2\right) \underbrace{\haus^2(S \cap B_r(o))}_{=\pi r^2 /2} \le \frac{\pi(|Du(o)|+Cr)}{1 +\pi r} \, \cdot
\end{align*}
In particular, $r\mapsto |\Delta_{\mu_2,r}^{\dist}u(o)|$ is bounded in a neighborhood of $0$.

We thus deduce that $\Delta_{\meas^{1}}^{\dist}u(o)$ exists if and only if
$$
\partial_x u(o)=0 \quad \text{and}  \quad a\df\lim\limits_{r\to 0^+} r^{-3}\int_0^r \partial_{xx}^2u(t) (r-t)^2/2 \di t \, \, \text{exists in $\setR$},
$$
in which case $\Delta_{\meas^{1}}^{\dist}u(o) = \Delta_{\meas_{1}^{1}}^{\dist}u(o)=2a$.

Performing similar calculations for $\mu^{3}$ shows that $\Delta_{\meas^{3}}^{\dist}u(o)$ exists if and only if $\partial_x u(o)=0$, in which case we have $\Delta_{\meas^{3}}^{\dist}u(o)= \Delta_{\meas_{2}}^{\dist}u(o)$. Note that in this case the main contribution to the AMV Laplacian comes from the $2$-dimensional piece $S$, while in the previous case it was coming from the $1$-dimensional piece $L$. In fact, this example shows that the Hausdorff dimension of a piece does not matter when one computes the AMV Laplacian at an intersection point: what really matters is the so-to-say Ahlfors regular dimension of the measures.

Finally in the case of $\mu^{2}$, explicit computations show that $\Delta_{\meas^{2}}^{\dist_e}u(o)$ exists if and only if $\partial_x u(o)=0$, in which case
$$\Delta_{\meas^{2}}^{\dist_e}u(o) = (1+\pi)^{-1} \Delta_{\meas_{1}^{2}}^{\dist_e}u(o) + \pi (1+\pi)^{-1} \Delta_{\meas_{2}}^{\dist_e}u(o),$$
as expected from Corollary \ref{cor:intersection-ahlfors-regular-measures-nondecreasing}.
\end{example}

\section{Maximum and Comparison Principles}\label{sec:maximum}

In this section we introduce the notion of AMV sub- and superharmonic functions, and 
show that an upper semicontinuous AMV subharmonic function attains its maximum at the boundary. We recall that a metric space is called proper whenever all closed subsets are compact, in which case any u.s.c. function defined on the closure of a bounded domain attains its maximum.

\begin{defn}
	[Pointwise Upper and Lower AMV Laplacian] Let $(X,\dist,\meas)$
	be a metric measure space and $u:X\to\mathbb{\overline{R}}$. Then we define the \emph{upper AMV Laplacian}
	$\overbar{\Delta}_{\meas}^{\dist}u$ and \emph{lower AMV Laplacian} $\uunderbar{\Delta}_{\meas}^{\dist}u$
	respectively as
	\begin{align*}
	\overbar{\Delta}_{\meas}^{\dist} & u(x)\df\varlimsup_{r\to0^{+}}\frac{1}{r^{2}}\fint_{B_{r}(x)}u(y)-u(x)\di\meas(y),\\
	\uunderbar{\Delta}_{\meas}^{\dist} & u(x)\df\varliminf_{r\to0^{+}}\frac{1}{r^{2}}\fint_{B_{r}(x)}u(y)-u(x)\di\meas(y),
	\end{align*}
	for $\mu$-a.e.~$x \in X$.
\end{defn}
Accordingly, for $\Omega\subseteq X$, the function $u\in L^1_\text{loc}(X,\mu)$ is called pointwise 
\begin{itemize}
	\item upper AMV subharmonic in $\Omega$ if $\overbar{\Delta}_{\meas}^{\dist}u(x)\ge0$ holds for all $x\in\Omega$,
	\item lower AMV subharmonic in $\Omega$ if $\uunderbar{\Delta}_{\meas}^{\dist}u(x)\ge0$ holds for all $x\in\Omega$,
	\item upper AMV superharmonic in $\Omega$ if $\overbar{\Delta}_{\meas}^{\dist}u(x)\le0$ holds for all $x\in\Omega$,
	\item lower AMV superharmonic in $\Omega$ if $\uunderbar{\Delta}_{\meas}^{\dist}u(x)\le0$ holds for all $x\in\Omega$.
\end{itemize}
We also add the word \emph{strictly} whenever the inequalities involved are strict.

\begin{lem}[Maximum Principle for Strictly Upper AMV Subharmonic Functions]\label{lem:strictmaxprinc}
	Let $(X,\dist,\meas)$ be a proper metric measure space and $u\in L_{\text{loc}}^{1}(\Omega,\meas)$
	an u.s.c. function in $\overline{\Omega}\subseteq X$ such that
	$\overbar{\Delta}_{\meas}^{\dist}u>0$ in $\Omega$. Then
	\[
	\max_{\partial\Omega}u=\max_{\overline{\Omega}}u,
	\]
	and $u$ does not attain its maximum in $\Omega$.
\end{lem}

\begin{proof}
	Let $C\df\{x\in\Omega:u(x)=\max_{\overline{\Omega}}u\}$. If $C$
	is empty we are done by the upper semicontinuity of $u$. Assume
	therefore that $C$ is nonempty and let $\max_{\overline{\Omega}}u=u(x)$
	for some $x\in C$. Note that $u(x)<\infty$ since otherwise $\Delta_{\meas,r}^{\dist}u(x)\equiv-\infty$
	for all $r>0$, which would contradict $\overbar{\Delta}_{\meas}^{\dist}u(x)>0$.
	Also, since $u\in L_{\text{loc}}^{1}(\Omega,\meas)$, $u(x)>-\infty$. Therefore
	$u(x)$ is finite and, for some $r$ small enough, we get the following
	contradiction,
	\[
	0<\Delta_{\meas,r}^{\dist}u(x)=\frac{1}{r^{2}}\fint_{B_{r}(x)}\underbrace{u(y)-u(x)}_{\le0}\di\meas(y)\le0.
	\]
	Therefore $C$ has to be empty.
\end{proof}

From this lemma we can deduce the weak maximum principle for pointwise upper AMV
subharmonic functions given the existence of a strictly
lower AMV subharmonic function.
\begin{thm}[Weak Maximum Principle for Upper AMV Subharmonic Functions] \label{thm:weak-maxprin-sub}
	Let $(X,\dist,\mu)$ be a proper metric measure space and
	$u\in L_{\text{loc}}^{1}(\Omega,\meas)$ be an u.s.c. function in $\overline{\Omega}\subseteq X$
	such that $\overbar{\Delta}_{\meas}^{\dist}u\ge0$ in
	$\Omega$. Assume that there exists a bounded function $\phi$
	which is u.s.c. in $\overline{\Omega}$ and such that $\uunderbar{\Delta}_{\meas}^{\dist}\phi>0$.
	Then $\max_{\partial\Omega}u=\max_{\overline{\Omega}}u$.
\end{thm}

\begin{proof}
	Assume that $\max_{\partial\Omega}u<\max_{\overline{\Omega}}u=u(x)$
	for some $x\in\Omega$. By the same argument as in Lemma \ref{lem:strictmaxprinc},
	$u(x)$ is finite. Let $M\df\lVert\phi\rVert_{L^{\infty}(\overline{\Omega})}$
	and take $\epsilon>0$ such that $\max_{\overline{\Omega}}u>\max_{\partial\Omega}u+2\epsilon M$.
	In particular, this implies that
	\begin{equation}
	\max_{\overline{\Omega}}u+\epsilon\inf_{\overline{\Omega}}\phi>\max_{\partial\Omega}u+\epsilon\max_{\partial\Omega}\phi.\label{eq:int>bound}
	\end{equation}
	Now define $u_{\epsilon}\df u+\epsilon\phi$. Then, pointwise in
	$\Omega$,
	\[
	\overbar{\Delta}_{\meas}^{\dist}u_{\epsilon}\ge\overbar{\Delta}_{\meas}^{\dist}u+\epsilon\uunderbar{\Delta}_{\meas}^{\dist}\phi>0,
	\]
	and since $u_{\epsilon}$ is u.s.c. in $\overline{\Omega}$, Lemma
	\ref{lem:strictmaxprinc} implies that $\max_{\partial\Omega}u_{\epsilon}=\max_{\overline{\Omega}}u_{\epsilon}$.
	However, by \eqref{eq:int>bound},
	\begin{align*}
	\max_{\overline{\Omega}}u_{\epsilon} & \ge u_{\epsilon}(x)=\max_{\overline{\Omega}}u+\epsilon\phi(x)\\
	& \ge\max_{\overline{\Omega}}u+\epsilon\inf_{\overline{\Omega}}\phi\\
	& >\max_{\partial\Omega}u+\epsilon\max_{\Omega}\phi\\
	& \ge\max_{\partial\Omega}u_{\epsilon},
	\end{align*}
	a contradiction. Hence $\max_{\partial\Omega}u=\max_{\overline{\Omega}}u$.
\end{proof}

\begin{rem}
As an example, in the Heisenberg group, we could for instance choose $\phi(x)$ as the graded coordinate function $x_1^2$.
\end{rem}

\begin{rem}
	By considering the sign function $\sgn x$ defined as zero for $x=0$, the function $u(x)=\sgn x - \sgn (x-1)$ in $\Omega\df(-1,2)$ is AMV harmonic everywhere in $\Omega$ but the maximum principle is violated, showing that upper semicontinuity is a necessary condition.\footnote{This example was proposed to us by A. Kijowski.}
\end{rem}

\begin{cor}[Weak Minimum Principle for Lower AMV Superharmonic Functions] \label{cor:weak-minprin-super}
	Let $(X,\dist,\mu)$ be a proper metric measure space and $u\in L_{\text{loc}}^{1}(\Omega,\meas)$ be a l.s.c. function in $\overline{\Omega}\subseteq X$
	such that $\uunderbar{\Delta}_{\meas}^{\dist}u\le0$ in $\Omega$.
	Assume that there exists a bounded function $\phi$ which is u.s.c.
	in $\overline{\Omega}$ and such that $\uunderbar{\Delta}_{\meas}^{\dist}\phi>0$.
	Then $\min_{\partial\Omega}u=\min_{\overline{\Omega}}u$.
\end{cor}

\begin{proof}
	Let $v\df-u$. Then $v$ is u.s.c. in $\overline{\Omega}$ and $\overbar{\Delta}_{\meas}^{\dist}v=\overbar{\Delta}_{\meas}^{\dist}(-u)=-\uunderbar{\Delta}_{\meas}^{\dist}u\ge0$.
	Hence we can conclude by applying Theorem \ref{thm:weak-maxprin-sub}
	to $v$.
\end{proof}
From Theorem \ref{thm:weak-maxprin-sub} and Corollary \ref{cor:weak-minprin-super}
we can deduce that continuous AMV harmonic functions attain extremal
values at the boundary.
\begin{cor}[Weak Max and Minimum Principle for AMV Harmonic Functions]
	Let $(X,\dist,\mu)$ be a proper metric measure space and $u\in L_{\text{loc}}^{1}(\Omega,\mu)$
	be a continuous function in $\overline{\Omega}\subseteq X$ such that
	$\Delta_{\meas}^{\dist}u=0$ in $\Omega$. Assume that there
	exists a bounded function $\phi$ which is u.s.c. in $\overline{\Omega}$
	and such that $\uunderbar{\Delta}_{\meas}^{\dist}\phi>0$. Then
	$\max_{\partial\Omega}u=\max_{\overline{\Omega}}u$ and $\min_{\partial\Omega}u=\min_{\overline{\Omega}}u$.
\end{cor}

Because of the superadditivity of $\overbar{\Delta}_{\mu}^{\dist}$
we also get the following comparison principle.
\begin{cor}[Weak Comparison Principle]
	Let $(X,\dist,\mu)$ be a proper metric measure space and $u,v\in L_{\text{loc}}^{1}(\Omega,\meas)$ be l.s.c. and u.s.c. functions
	respectively in $\overline{\Omega}\subseteq X$ such that $u\ge v$ on $\partial\Omega$, and either
	\begin{enumerate}[label=(\roman*)]
		\item 	$\uunderbar{\Delta}_{\meas}^{\dist}u\le0$ and $\uunderbar{\Delta}_{\meas}^{\dist}v\ge0$ in $\Omega$, or
		\item $\overbar{\Delta}_{\meas}^{\dist}u\le0$ and $\overbar{\Delta}_{\meas}^{\dist}v\ge0$ in $\Omega$.
	\end{enumerate}
	Assume also that there exists a bounded
	function $\phi$ which is u.s.c. in $\overline{\Omega}$ and such
	that $\uunderbar{\Delta}_{\meas}^{\dist}\phi>0$. Then $u\ge v$
	in $\overline{\Omega}$.
\end{cor}

\begin{proof}
	We assume \textit{(i)} and note that $w\df u-v$ is l.s.c.~in $\overline{\Omega}$ and $\uunderbar{\Delta}_{\meas}^{\dist}w\le\uunderbar{\Delta}_{\meas}^{\dist}u-\uunderbar{\Delta}_{\meas}^{\dist}v\le0$.
	Therefore $\min_{\overline{\Omega}}w=\min_{\partial\Omega}w\ge0$
	by Corollary \ref{cor:weak-minprin-super}, hence $u\ge v$ in $\overline{\Omega}$.
	
	If we instead assume \textit{(ii)}, we get $\uunderbar{\Delta}_{\meas}^{\dist}w\le\overbar{\Delta}_{\meas}^{\dist}u-\overbar{\Delta}_{\meas}^{\dist}v\le0$,
	and Corollary~\ref{cor:weak-minprin-super} again gives the conclusion. 
\end{proof}
\begin{rem}
	In particular, if $u$ and $v$ are continuous AMV harmonic functions
	such that $u\ge v$ on the boundary $\partial\Omega$, and $\phi$
	as described above exists, then $u\ge v$ in $\overline{\Omega}$.
\end{rem}

\section{Green-Type Identity and Weak AMV Laplacian}


A metric measure space $(X,\dist,\meas)$ being given, we set
$$
T_ru := \fint_{B_r(x)}u\di \mu
$$
for any $u \in L^1_{\text{loc}}(X,\mu)$, $r>0$ and $x \in X$. Let us provide a preliminary result.

\begin{lem}\label{lem:technical}
Let $(X,\dist,\mu)$ be a metric measure space. Set $w(x)\df\int_{B_r(x)}\frac{\di \mu(y)}{\mu(B_r(y))}$ for any $x \in X$. Then for any $1 \le p \le +\infty$ and $r>0$,
\begin{enumerate}
\item if $u \in L^p(X,w\mu)$ then $T_ru \in L^p(X,\mu)$ and
$$
\|T_r u\|_{L^p(X,\mu)} \le \|u\|_{L^p(X,w \mu)} \, ,
$$
\item if $u \in L^p(X,\mu) \cap L^p(X,w \mu)$ then $\Delta_{\mu,r}^{\dist}u \in L^p(X,\mu)$ and
$$
\|\Delta_{\mu,r}^{\dist}u\|_{L^p(X,\mu)} \le \frac{1}{r^2}(\|u\|_{L^p(X,\mu)} + \|u\|_{L^p(X,w\mu)})\, .
$$
\end{enumerate}
\end{lem}

\begin{proof}
Let us start with proving $\textit{1.}$ In case $p=+\infty$ the result is immediate. For $1\le p <+\infty$, if $u \in L^p(X,w\mu)$, Jensen's inequality with the convex function $t \mapsto |t|^p$ implies
\begin{align*}
\|T_r u\|_{L^p(X,\mu)}^p & \le \int_X \fint_{B_r(x)} |u(y)|^p \di \mu(y) \di \mu(x)\\
& = \int_X \int_X \frac{\chi_{B_r(x)}(y)}{\mu(B_r(x))}|u(y)|^p \di \mu(y) \di \mu(x).
\end{align*}
Applying the Fubini--Tonelli theorem and the simple observation that $\chi_{B_r(x)}(y) = \chi_{B_r(y)}(x)$, we get
\begin{align*}
\|T_r u\|_{L^p(X,\mu)}^p & \le \int_X |u(y)|^p \int_X \frac{\chi_{B_r(y)}(x)}{\mu(B_r(x))} \di \mu(x) \di \mu(y)\\
& = \int_X |u(y)|^p \int_{B_r(y)} \frac{1}{\mu(B_r(x))} \di \mu(x) \di \mu(y)\\
& = \int_X |u(y)|^p w(y) \di \mu(y)= \|u\|_{L^p(X,w\mu)}^p\,.
\end{align*}
This shows $\textit{1}$. That $\textit{2}.$ follows is straightforward: for any $1\le p \le +\infty$, if $u \in L^p(X,\mu) \cap L^p(X,w\mu)$, then
\begin{align*}
\|\Delta_{\mu,r}^{\dist}u\|_{L^p(X,\mu)} \le \frac{1}{r^2} \|T_r u - u\|_{L^p(X,\mu)}&\le \frac{1}{r^2}\left( \|T_r u\|_{L^p(X,\mu)} + \|u\|_{L^p(X,\mu)}\right)\\
& \le \frac{1}{r^2}\left( \|u\|_{L^p(X,w\mu)} + \|u\|_{L^p(X,\mu)}\right)
\end{align*}
due to $\textit{1.}$
\end{proof}

\begin{rem}
Lemma \ref{lem:technical} applies especially when $\mu$ is a doubling measure, in which case the condition $u\in L^p(X,w\mu)$ is superfluous. Explicit computations show that if $C_\mu$ is the doubling measure of $\mu$, then $\|T_r\|_{p\to p}\le C_\mu^2$ and $\|\Delta_{\mu,r}^{\dist}\|_{p\to p} \le (1+C_\mu^2)r^{-2}$. For the particular case of a $Q$-Ahlfors regular measure, one has $\|T_r\|_{p\to p}\le C/c$ and $\|\Delta_{\mu,r}^{\dist}\|_{p\to p} \le (1+C/c)r^{-2}$. Finally, in case $\mu$ is uniform, meaning that there exists $\omega>0$ and $Q\ge 0$ such that $\mu(B)=\omega r^Q$ for any ball $B$ with radius $r$, one has $\|T_r\|_{p\to p}\le 1$ and $\|\Delta_{\mu,r}^{\dist}\|_{p\to p} \le 2r^{-2}$.
\end{rem}

\begin{thm}[Green-Type Identity]\label{thm:GaussGreen}
Let $(X,\dist,\mu)$ be a metric measure space, and $w$ be as in Lemma \ref{lem:technical}. Then for any $u,v \in L^2(X,w\mu)\cap  L^2(X,\mu)$ and $r>0$,
\begin{align}\label{eq:GaussGreen}
\int_X v \Delta_{\mu,r}^{\dist} u - u \Delta_{\mu,r}^{\dist} v \di \mu & = \frac{1}{r^2} \int_X u(x) \fint_{B_r(x)} v(y) \left( \frac{\mu(B_r(x))}{\mu(B_r(y))} - 1\right) \di \mu(y) \di \mu(x) \nonumber \\
& = \frac{1}{r^2} \int_X u(x) \int_{B_r(x)} v(y) \left( \frac{1}{\mu(B_r(y))} - \frac{1}{\mu(B_r(x))}\right) \di \mu(y) \di \mu(x).
\end{align}
\end{thm}

\begin{proof}
By 2. in Lemma~\ref{lem:technical}, the integral on the left hand side exists. Moreover,
\begin{equation}\label{eq:ggdifference}
 r^{2}\int_{X}v\Delta_{\meas,r}^{\dist}u-u\Delta_{\meas,r}^{\dist}v\di\meas
= \int_{X}\fint_{B_{r}(x)}u(y)v(x)-v(y)u(x)\di\meas(y)\di\meas(x).
\end{equation}
Again by Fubini--Tonelli and the fact that $\chi_{B_r(x)}(y) = \chi_{B_r(y)}(x)$,
\begin{align*}
 \int_{X}\fint_{B_{r}(x)}u(y)v(x)\di\meas(y)\di\meas(x) & =\int_{X}\int_{X}\frac{\chi_{B_{r}(x)}(y)}{\meas(B_{r}(x))}u(y)v(x)\di\meas(y)\di\meas(x)\\
 & =\int_{X}\int_{X}\frac{\chi_{B_{r}(x)}(y)}{\meas(B_{r}(x))}u(y)v(x)\di\meas(x)\di\meas(y)\\
 & =\int_{X}\int_{X}\frac{\chi_{B_{r}(y)}(x)}{\meas(B_{r}(x))}u(y)v(x)\di\meas(x)\di\meas(y).
\end{align*}
By relabeling $x$ to $y$ and vice versa in the final expression above, we find that
\begin{align*}
\int_{X}\int_{X}\frac{\chi_{B_{r}(y)}(x)}{\meas(B_{r}(x))}u(y)v(x)\di\meas(x)\di\meas(y)&=\int_{X}\int_{X}\frac{\chi_{B_{r}(x)}(y)}{\meas(B_{r}(y))}u(x)v(y)\di\meas(y)\di\meas(x) \\
&=\int_{X}\fint_{B_{r}(x)}\frac{\meas(B_{r}(x))}{\meas(B_{r}(y))}u(x)v(y)\di\meas(y)\di\meas(x).
\end{align*}
By putting this in \eqref{eq:ggdifference}, we conclude that
\begin{align*}
r^{2}\int_{X}v\Delta_{\meas,r}^{\dist}u-u\Delta_{\meas,r}^{\dist}v\di\meas & =\int_{X}\fint_{B_{r}(x)}u(x)v(y)\left(\frac{\meas(B_{r}(x))}{\meas(B_{r}(y))}-1\right)\di\meas(y)\di\meas(x)\\
 & =\int_{X}\int_{B_{r}(x)}u(x)v(y)\left(\frac{1}{\meas(B_{r}(y))}-\frac{1}{\meas(B_{r}(x))}\right)\di\meas(y)\di\meas(x).
\end{align*}
\end{proof}

\begin{rem}
The theorem shows that $\Delta_{\meas,r}^{\dist}$ in general is not self-adjoint on $L^2(X,\mu)$. However for measures such that $\mu(B_r(x))/\mu(B_r(y)) = 1 + o(r^2)$ uniformly, in particular uniform measures, the right-hand side of \eqref{eq:GaussGreen} is zero. This can be compared to the result of Burago et al. \cite{MR3990939} where they show that $\Delta_{\meas,r}^{\dist}$ is self-adjoint on $L^2(X,\phi\mu)$, where $\phi(x)\df r^2 \meas(B_r(x))$.
\end{rem}

If one seeks for a weak definition of the AMV Laplacian, Theorem \ref{thm:GaussGreen} suggests to avoid ``differentiation'' of the test functions as one would naturally do. Indeed, in \eqref{eq:GaussGreen}, there is no a priori reason to get $\int_X v \Delta_{\mu,r}^{\dist} u - u \Delta_{\mu,r}^{\dist} v \di \mu \to 0$ when $r\to 0^+$. Moreover, the regularity of test functions in metric spaces do not guarantee the existence of the AMV Laplacian in any case. Therefore we propose the following definition.

\begin{defn}
	[Weak AMV Laplacian]\label{def:weakAMV} Let $(X,\dist,\meas)$ be a metric measure
	space.
	We say that a Borel measure $\nu$ on $X$ is the \emph{weak AMV Laplacian} of a $\mu$-measurable function $u:X\to\mathbb{\overline{R}}$, and we denote it by $\Delta_{\meas}^{\dist}u\stackrel{{\rm w}}{=}\nu$, if
	\[
	\lim_{r\to0^{+}}\int_{X}\phi(x)\Delta_{\meas,r}^{\dist}u(x)\di\meas(x)=\int_{X}\phi(x)\di\nu(x)
	\]
	holds for all $\phi\in C_{c}(X)$. When $\Delta_{\meas}^{\dist}u\stackrel{{\rm w}}{=}0$ we say that $u$ is \emph{weakly AMV harmonic}.
\end{defn}

Note that a function
	which is pointwise AMV harmonic might not be weakly AMV harmonic. For instance,
	for $(X,\dist,\meas)=(\mathbb{R},\dist_e,\mathcal{L}^1)$, the sign function $u(x)=\sgn x$ defined as zero at the origin
satisfies $\Delta_{\meas}^{\dist}u=0$ everywhere in $\setR$,
	but a straightforward computation shows that
	\[
	\Delta_{\meas}^{\dist}u\stackrel{{\rm w}}{=}\frac{2\delta'}{6},
	\]
 which coincides with the distributional
	Laplacian of $u$ divided by the dimensional constant $2(n+2)$ for $n=1$.


\bibliographystyle{abbrvamsalpha}
\bibliography{References}

\providecommand{\bysame}{\leavevmode\hbox to3em{\hrulefill}\thinspace}
\providecommand{\MR}{\relax\ifhmode\unskip\space\fi MR }
\providecommand{\MRhref}[2]{%
  \href{http://www.ams.org/mathscinet-getitem?mr=#1}{#2}
}
\providecommand{\href}[2]{#2}
\begin{thebibliography}{ABGLBO19}

\bibitem[ABGLBO19]{MR3985395}
G.~Alberti, G.~Buttazzo, S.~Guarino Lo~Bianco, and E.~Oudet, \emph{Optimal
  reinforcing networks for elastic membranes}, Netw. Heterog. Media \textbf{14}
  (2019), no.~3, 589--615.

\bibitem[AGG19]{MR3896674}
T.~Adamowicz, M.~Gaczkowski, and P.~G\'{o}rka, \emph{Harmonic functions on
  metric measure spaces}, Rev. Mat. Complut. \textbf{32} (2019), no.~1,
  141--186.

\bibitem[AKS]{aks}
T.~Adamowicz, A.~Kijowski, and T.~Soultanis, \emph{in preparation}.

\bibitem[BIK19]{MR3990939}
D.~Burago, S.~Ivanov, and Y.~Kurylev, \emph{Spectral stability of
  metric-measure {L}aplacians}, Israel J. Math. \textbf{232} (2019), no.~1,
  125--158.

\bibitem[BLU07]{MR2363343}
A.~Bonfiglioli, E.~Lanconelli, and F.~Uguzzoni, \emph{Stratified {L}ie groups
  and potential theory for their sub-{L}aplacians}, Springer Monographs in
  Mathematics, Springer, Berlin, 2007.

\bibitem[Bos65]{MR177128}
A.~K. Bose, \emph{Functions satisfying a weighted average property}, Trans.
  Amer. Math. Soc. \textbf{118} (1965), 472--487.

\bibitem[FLM14]{MR3177660}
F.~Ferrari, Q.~Liu, and J.~Manfredi, \emph{On the characterization of
  {$p$}-harmonic functions on the {H}eisenberg group by mean value properties},
  Discrete Contin. Dyn. Syst. \textbf{34} (2014), no.~7, 2779--2793.

\bibitem[Gau40]{gauss}
C.~F. Gauss, \emph{Allgemeine lehrs{\"a}tze in beziehung auf die verkehrten
  verh{\"a}ltnisse des quadrats der entfernung wirkenden anziehungs- und
  abstossungs-kr{\"a}fte}, Weidmannschen Buchhandlung (1840).

\bibitem[GG09]{MR2545982}
M.~Gaczkowski and P.~G\'{o}rka, \emph{Harmonic functions on metric measure
  spaces: convergence and compactness}, Potential Anal. \textbf{31} (2009),
  no.~3, 203--214.

\bibitem[Gra73]{MR0339002}
A.~Gray, \emph{The volume of a small geodesic ball of a {R}iemannian manifold},
  Michigan Math. J. \textbf{20} (1973), 329--344.

\bibitem[Kel34]{kellogg}
O.~D. Kellogg, \emph{Converses of {G}auss' theorem on the arithmetic mean},
  Trans. Amer. Math. Soc. \textbf{36} (1934), no.~2, 227--242.

\bibitem[Kij18]{kijowski}
A.~Kijowski, \emph{Characterization of mean value harmonic functions on norm
  induced metric measure spaces with weighted {L}ebesgue measure}, ArXiV
  preprint: 1804.10005v2 (2018).

\bibitem[Koe06]{koebe}
P.~Koebe, \emph{Herleitung der partiellen differentialgleichungen der
  potentialfunktion aus deren integraleigenschaft}, Sitzungsber. Berlin. Math.
  Gessellschaft \textbf{5} (1906), 39--42.

\bibitem[KP89]{MR967797}
L.~Karp and M.~Pinsky, \emph{Volume of a small extrinsic ball in a
  submanifold}, Bull. London Math. Soc. \textbf{21} (1989), no.~1, 87--92.

\bibitem[Llo15]{MR3299033}
J.~G. Llorente, \emph{Mean value properties and unique continuation}, Commun.
  Pure Appl. Anal. \textbf{14} (2015), no.~1, 185--199.

\bibitem[MPR10]{MR2566554}
J.~J. Manfredi, M.~Parviainen, and J.~D. Rossi, \emph{An asymptotic mean value
  characterization for {$p$}-harmonic functions}, Proc. Amer. Math. Soc.
  \textbf{138} (2010), no.~3, 881--889. \MR{2566554}

\bibitem[NV94]{MR1321628}
I.~Netuka and J.~Vesel\'{y}, \emph{Mean value property and harmonic functions},
  Classical and modern potential theory and applications ({C}hateau de {B}onas,
  1993), NATO Adv. Sci. Inst. Ser. C Math. Phys. Sci., vol. 430, Kluwer Acad.
  Publ., Dordrecht, 1994, pp.~359--398.

\bibitem[Vol09]{volterra}
V.~Volterra, \emph{Alcune osservazioni sopra propriet{\`a} atte ad individuare
  una funzione}, Rend. Accad. d. Lincei Roma \textbf{18} (1909), no.~5,
  263--266.

\bibitem[WG82]{MR677493}
T.~J. Willmore and A.~Gray, \emph{Mean-value theorems for {R}iemannian
  manifolds}, Proc. Roy. Soc. Edinburgh Sect. A \textbf{92} (1982), no.~3-4,
  343--364.

\end{thebibliography}

\end{document}